\documentclass[11pt]{article}
\usepackage[margin=1.1in]{geometry}
\usepackage{amsmath,amssymb,amsthm}

\usepackage[colorlinks=true,allcolors=blue]{hyperref}

\newtheorem{theorem}{Theorem}[section]
\newtheorem{lemma}[theorem]{Lemma}
\newtheorem{proposition}[theorem]{Proposition}
\newtheorem{corollary}[theorem]{Corollary}
\newtheorem{conjecture}[theorem]{Conjecture}
\theoremstyle{definition}
\newtheorem{definition}[theorem]{Definition}
\theoremstyle{remark}
\newtheorem{remark}[theorem]{Remark}

\newcommand{\R}{\mathbb{R}}
\newcommand{\N}{\mathbb{N}}
\newcommand{\supp}{\operatorname{supp}}
\newcommand{\diam}{\operatorname{diam}}

\newcommand{\Dm}{D_\mu}

\title{Square Functions and Rectifiability\\
under Monotone Transformations of the Density}

\author{Triet M.\ Le\thanks{Spatiolyx LLC; \href{mailto:tml@spatiolyx.ai}{\texttt{tml@spatiolyx.ai}}}}

\begin{document}

\maketitle

\begin{abstract}
Let $\mu$ be an $n$-AD-regular measure in $\R^d$. Chousionis, Garnett, Le and Tolsa
\cite{CGLT} proved that $\mu$ is uniformly $n$-rectifiable if and only if the square
function built from the density differences
$\Delta_\mu(x,r)=\mu(B(x,r))/r^n-\mu(B(x,2r))/(2r)^n$ satisfies a Carleson condition.
In this paper we show that the same characterization holds if the density is first
composed with a function $F$ which is bi-Lipschitz on the interval
$[c_0^{-1},c_0]$ determined by the AD-regularity constant $c_0$. The main example is
$F=\log$, introduced
in \cite{Le-URJEPA}, for which the square function takes the scale-invariant form
$\Delta_\mu^{\log}(x,r) = \log\bigl(\mu(B(x,r))/\mu(B(x,2r))\bigr)+n\log 2$. We give a
complete proof, extend the statement to the smooth square functions of \cite{CGLT},
where the density is replaced by the convolution of $\mu$ with a Gaussian or a more
general radial kernel, discuss what happens when $F$ is not bi-Lipschitz, and treat the
case $\mu(\R^d)<\infty$, where the behavior of $F$ near zero enters in only one of the
two implications. We also show that the qualitative characterization of
$n$-rectifiable measures by Tolsa and Toro \cite{TT}, in terms of the same square function at
$\mu$-almost every point, holds after composition with any locally bi-Lipschitz $F$.
This requires neither AD-regularity nor doubling, and for $F=\log$ the condition
$\lim_{r\to0}\Delta_\mu(x,r)=0$ becomes $\lim_{r\to0}\mu(B(x,r))/\mu(B(x,2r))=2^{-n}$.
\end{abstract}

\section{Introduction and statement of results}

Throughout the paper $n$ and $d$ are integers with $0<n<d$. A Radon measure $\mu$ in $\R^d$ is
\emph{$n$-dimensional Ahlfors--David regular} ($n$-AD-regular) with constant $c_0\ge 1$ if
\begin{equation}\label{eq:ADR}
c_0^{-1}r^n \le \mu(B(x,r)) \le c_0 r^n
\qquad\text{for all } x\in\supp(\mu),\ 0<r\le\diam(\supp(\mu)).
\end{equation}
We always assume that
$\diam(\supp(\mu))>0$, that is, the $\supp(\mu)$ is not a single point, since
otherwise \eqref{eq:ADR} is trivial.
Uniform $n$-rectifiability is understood in the sense of David and Semmes
\cite{DS1,DS2}: there exist $\theta,M>0$ such that for all $x\in\supp(\mu)$ and
$0<r\le\diam(\supp(\mu))$ there is a Lipschitz map
$\rho\colon B_n(0,r)\subset\R^n\to\R^d$ with $\operatorname{Lip}(\rho)\le M$ and
$\mu\bigl(B(x,r)\cap\rho(B_n(0,r))\bigr)\ge\theta r^n$. The restriction
$r\le\diam(\supp(\mu))$ agrees with \eqref{eq:ADR}. If $\mu(\R^d)=\infty$, then
$\diam(\supp(\mu))=\infty$, so this imposes no upper restriction on $r$. If
$\supp(\mu)$ is bounded, the restriction is necessary, since
$\mu(B(x,r))\le\mu(\R^d)<\theta r^n$ for sufficiently large $r$.
We denote
\[
\Dm(x,r):=\frac{\mu(B(x,r))}{r^n},
\qquad
\Delta_\mu(x,r):=\Dm(x,r)-\Dm(x,2r).
\]

Recall the following theorem of Chousionis, Garnett, Le and Tolsa \cite{CGLT}. It is
stated there for all $n$-AD-regular measures, with no restriction on the total mass;
the case $\mu(\R^d)<\infty$ is discussed in \cite[Section 2.1]{CGLT}. 

\begin{theorem}[{\cite[Theorem 1.1]{CGLT}}]\label{thm:CGLT}
Let $\mu$ be an $n$-AD-regular measure in $\R^d$. Then $\mu$ is uniformly
$n$-rectifiable if and only if there exists a constant $c$ such that, for any ball
$B(x_0,R)$ centered at $\supp(\mu)$,
\begin{equation}\label{eq:carleson}
\int_0^R\!\!\int_{x\in B(x_0,R)} |\Delta_\mu(x,r)|^2\, d\mu(x)\,\frac{dr}{r}
\;\le\; c\,R^n.
\end{equation}
\end{theorem}

In this paper we consider what happens when the density is composed with a monotone
function $F$ before taking the difference. The main example is $F=\log$. In this case
the resulting quantity depends only on the ratio $\mu(B(x,r))/\mu(B(x,2r))$ and not on
the normalization by $r^n$, which makes it a natural quantity to consider for doubling
measures which are not AD-regular; see for instance \cite{ADT} and \cite{To-memoir}.
This scale-invariance is also what allows the logarithmic square function to be used as
a regularizer on learned representations in \cite{Le-URJEPA}, where
$\Delta^{\log}_\mu$ depends only on the ratio $\mu(B(x,r))/\mu(B(x,2r))$.

\begin{definition}\label{def:DeltaF}
Let $F\colon(0,\infty)\to\R$ be a Borel function. For $x\in\supp(\mu)$ and $r>0$ set
\[
\Delta^F_\mu(x,r):=F\bigl(\Dm(x,r)\bigr)-F\bigl(\Dm(x,2r)\bigr).
\]
For $x\notin\supp(\mu)$ we set $\Delta^F_\mu(x,r):=0$; since
$\mu(\R^d\setminus\supp(\mu))=0$, this convention does not affect any of the integrals
below. Given a compact interval $J\subset(0,\infty)$ and $\lambda\ge 1$, we say $F$ is
\emph{bi-Lipschitz on $J$ with constant $\lambda$} if
\[
\lambda^{-1}|s-t|\;\le\;|F(s)-F(t)|\;\le\;\lambda|s-t|
\qquad\text{for all } s,t\in J.
\]
We do not assume that $F$ is monotone; see Remark~\ref{rem:scope}.
\end{definition}

Note that the convention $\Delta^F_\mu=0$ outside $\supp(\mu)$ is needed: for
$x\notin\supp(\mu)$ and $r$ small enough we have $\Dm(x,r)=0$, and $F$ (for instance
$F=\log$) need not be defined at $0$.

Our main results are the following. Throughout, $J_0:=[c_0^{-1},c_0]$ denotes the
interval to which \eqref{eq:ADR} confines the density at admissible scales.

\begin{theorem}\label{thm:main}
Let $\mu$ be an $n$-AD-regular measure in $\R^d$ with constant $c_0$ and
$\mu(\R^d)=\infty$. Let $F\colon(0,\infty)\to\R$ be a Borel function and
bi-Lipschitz on $J_0$ with constant $\lambda$. Then $\mu$ is uniformly $n$-rectifiable if and only if
there exists a constant $c$ such that, for any ball $B(x_0,R)$ centered at $\supp(\mu)$,
\begin{equation}\label{eq:carlesonF}
\int_0^R\!\!\int_{x\in B(x_0,R)} |\Delta^F_\mu(x,r)|^2\, d\mu(x)\,\frac{dr}{r}
\;\le\; c\,R^n.
\end{equation}
Quantitatively: if \eqref{eq:carleson} holds with constant $c'$ then
\eqref{eq:carlesonF} holds with constant $\lambda^2 c'$, and if \eqref{eq:carlesonF}
holds with constant $c$ then \eqref{eq:carleson} holds with constant $\lambda^2 c$.
\end{theorem}

\begin{corollary}[logarithmic densities]\label{cor:log}
Let $F=\log$, so that
\[
\Delta^{\log}_\mu(x,r)
=\log\frac{\mu(B(x,r))}{\mu(B(x,2r))}+n\log 2,
\]
then Theorem~\ref{thm:main} holds with $\lambda=c_0$. Moreover the hypothesis
$\mu(\R^d)=\infty$ can be removed, with different constants; see
Proposition~\ref{prop:bounded}.
\end{corollary}

In \cite{CGLT} the density $\Dm(x,r)$ is also replaced by the convolution of $\mu$ with
a smooth radial kernel, and the same characterization is obtained. We recall the
setting. Let $\phi\colon\R^d\to\R$ be either of the form $\phi(x)=e^{-|x|^{2N}}$
with $N\in\N$, or $\phi(x)=(1+|x|^2)^{-a}$ with $a>n/2$. For $t>0$ let
\[
\phi_t(x):=\frac1{t^n}\,\phi\Bigl(\frac{x}{t}\Bigr),
\qquad
D_{\mu,\phi}(x,t):=\phi_t*\mu(x)=\int\phi_t(y-x)\,d\mu(y),
\]
so that $D_{\mu,\phi}(x,t)$ is a smooth version of $\Dm(x,t)$. Following
\cite{CGLT}, we set
\[
\Delta_{\mu,\phi}(x,t):=\int\bigl(\phi_t(y-x)-\phi_{2t}(y-x)\bigr)\,d\mu(y)
=D_{\mu,\phi}(x,t)-D_{\mu,\phi}(x,2t),
\]
and, with $\partial_\phi(x,t):=t\,\partial_t\phi_t(x)$,
\[
\widetilde\Delta_{\mu,\phi}(x,t):=\int\partial_\phi(y-x,t)\,d\mu(y)
=t\,\partial_t D_{\mu,\phi}(x,t).
\]
Note that $\phi_t(x)-\phi_{2t}(x)$ is a discrete approximation to
$\partial_\phi(x,t)$, and hence $\Delta_{\mu,\phi}(x,t)$ is a discrete
approximation to $\widetilde\Delta_{\mu,\phi}(x,t)$. By \cite[Theorem 1.2]{CGLT}, an $n$-AD-regular
measure $\mu$ is uniformly $n$-rectifiable if and only if the Carleson condition
\eqref{eq:carleson} holds with $\Delta_\mu$ replaced by $\Delta_{\mu,\phi}$, and if
and only if it holds with $\Delta_\mu$ replaced by $\widetilde\Delta_{\mu,\phi}$.
Given a Borel function $F\colon(0,\infty)\to\R$, we set, for $x\in\supp(\mu)$
and $t>0$,
\[
\Delta^F_{\mu,\phi}(x,t):=F\bigl(D_{\mu,\phi}(x,t)\bigr)-F\bigl(D_{\mu,\phi}(x,2t)\bigr),
\qquad
\widetilde\Delta^F_{\mu,\phi}(x,t):=t\,\partial_t\,F\bigl(D_{\mu,\phi}(x,t)\bigr),
\]
the latter when $F\in C^1(0,\infty)$, and $\Delta^F_{\mu,\phi}=\widetilde\Delta^F_{\mu,\phi}=0$
for $x\notin\supp(\mu)$. For $F=\log$ one has
$\widetilde\Delta^{\log}_{\mu,\phi}=\widetilde\Delta_{\mu,\phi}/D_{\mu,\phi}$,
the logarithmic derivative of $t\mapsto\phi_t*\mu(x)$. In this setting the interval
$J_0$ is replaced by
\begin{equation}\label{eq:Jphi}
J_\phi:=\bigl[\phi(1)\,c_0^{-1},\;C_\phi\,c_0\bigr],
\qquad
C_\phi:=\phi(0)+\sum_{k\ge1}2^{kn}\phi(2^{k-1}),
\end{equation}
where we write $\phi(\rho)$ for the common value of $\phi$ on $\{|x|=\rho\}$.
Note that $C_\phi<\infty$ for both classes of kernels, and that for
$\phi(x)=(1+|x|^2)^{-a}$ this is exactly where $a>n/2$ is used.

\begin{theorem}\label{thm:smooth}
Let $\mu$ be an $n$-AD-regular measure in $\R^d$ with constant $c_0$ and
$\mu(\R^d)=\infty$, let $\phi$ be as above, and let $F\colon(0,\infty)\to\R$ be
a Borel function and bi-Lipschitz on $J_\phi$ with constant $\lambda_\phi$. The
following are equivalent:
\begin{itemize}
\item[(a)] $\mu$ is uniformly $n$-rectifiable.
\item[(b)] There exists a constant $c$ such that for any ball $B(x_0,R)$ centered at
$\supp(\mu)$,
\begin{equation}\label{eq:carlesonFphi}
\int_0^R\!\!\int_{x\in B(x_0,R)}|\Delta^F_{\mu,\phi}(x,t)|^2\,d\mu(x)\,\frac{dt}{t}
\;\le\;c\,R^n.
\end{equation}
\end{itemize}
If in addition $F\in C^1(0,\infty)$, then (a) and (b) are also equivalent to (c).
\begin{itemize}
\item[(c)] There exists a constant $c$ such that for any ball $B(x_0,R)$ centered at
$\supp(\mu)$,
\begin{equation}\label{eq:carlesonFphitilde}
\int_0^R\!\!\int_{x\in B(x_0,R)}|\widetilde\Delta^F_{\mu,\phi}(x,t)|^2\,d\mu(x)\,\frac{dt}{t}
\;\le\;c\,R^n.
\end{equation}
\end{itemize}
The Carleson constants in (b) and (c) differ from those of the corresponding conditions
for $\Delta_{\mu,\phi}$ and $\widetilde\Delta_{\mu,\phi}$ by at most a factor
$\lambda_\phi^2$. For $F=\log$ one can take $\lambda_\phi=\max\bigl(C_\phi
c_0,\,c_0/\phi(1)\bigr)$.
\end{theorem}

The proof of Theorem~\ref{thm:main} is simple. By AD-regularity the density
$\Dm(x,r)$ stays in the compact interval $J_0$, on which $F$ is bi-Lipschitz, so that
$|\Delta^F_\mu|$ and $|\Delta_\mu|$ are comparable pointwise and the two Carleson
conditions are equivalent. In other words, Theorem~\ref{thm:main} is a change of
variables in Theorem~\ref{thm:CGLT} and not a new characterization. We write down the
statement and the proof here for two reasons: for $F=\log$ the quantity
$\Delta^{\log}_\mu$ is the natural one once AD-regularity is dropped, and the examples
in Section~\ref{sec:degenerate} show what can and cannot be obtained by this argument
when $F$ is not bi-Lipschitz.

The paper is organized as follows. In Section~\ref{sec:proof} we prove
Theorem~\ref{thm:main} and Corollary~\ref{cor:log}. In Section~\ref{sec:smooth} we
prove Theorem~\ref{thm:smooth}, and show that for $F$ with bounded derivatives up
to order $k$ the characterization also holds with the higher order square functions
$t^k\partial_t^kF(\phi_t*\mu)$ of \cite[Proposition 1.3]{CGLT}
(Proposition~\ref{prop:higher}). In Section~\ref{sec:bounded} we
treat the case $\mu(\R^d)<\infty$, where the large scales have to be handled separately
and the two implications require different hypotheses on $F$. In
Section~\ref{sec:degenerate} we discuss continuous strictly monotone $F$ which are not
bi-Lipschitz on $J_0$. In Section~\ref{sec:qualitative} we leave the AD-regular class
and show that the qualitative characterization of $n$-rectifiable measures of Tolsa and
Toro \cite{TT} also holds after composition with a locally bi-Lipschitz $F$.
The proofs of Propositions~\ref{prop:higher} and~\ref{prop:bounded} are given in
Appendices~\ref{app:higher} and~\ref{app:bounded}, and Appendix~\ref{app:numerics}
discusses the advantages of $\Delta^{\log}_\mu$ over $\Delta_\mu$ in numerical
computation.

Throughout the paper the letter $C$ stands for some constant which may change its
value at different occurrences. The notation $A\lesssim B$ means that there is some
fixed constant $C$ such that $A\le CB$, with $C$ as above. Also, $A\approx B$ is
equivalent to $A\lesssim B\lesssim A$.

\section{Proof of the main theorem}\label{sec:proof}

We start with two elementary lemmas.

\begin{lemma}\label{lem:confinement}
Let $\mu$ be $n$-AD-regular with constant $c_0$ and $\mu(\R^d)=\infty$, and let
$J_0:=[c_0^{-1},c_0]$. Then for all $x\in\supp(\mu)$ and all $r>0$,
\[
\Dm(x,r)\in J_0 \quad\text{and}\quad \Dm(x,2r)\in J_0.
\]
\end{lemma}

\begin{proof}
If $\supp(\mu)$ were bounded, then it would be compact, and since $\mu$ is a Radon
measure and $\mu(\R^d\setminus\supp(\mu))=0$, we would have
$\mu(\R^d)=\mu(\supp(\mu))<\infty$, a contradiction. Hence
$\diam(\supp(\mu))=\infty$, so \eqref{eq:ADR} applies at every radius $r>0$.
Applying it at the radii $r$ and $2r$ gives the claim.
\end{proof}

\begin{lemma}[pointwise comparability]\label{lem:pointwise}
Under the hypotheses of Theorem~\ref{thm:main}, for all $x\in\supp(\mu)$ and $r>0$,
\begin{equation}\label{eq:comparable}
\lambda^{-1}\,|\Delta_\mu(x,r)|
\;\le\;|\Delta^F_\mu(x,r)|
\;\le\;\lambda\,|\Delta_\mu(x,r)|.
\end{equation}
\end{lemma}

\begin{proof}
Fix $x\in\supp(\mu)$ and $r>0$, and write
\[
s:=\Dm(x,r)=\frac{\mu(B(x,r))}{r^n},
\qquad
t:=\Dm(x,2r)=\frac{\mu(B(x,2r))}{(2r)^n}.
\]
With this notation, by Definition~\ref{def:DeltaF} and the definition of $\Delta_\mu$,
\[
\Delta_\mu(x,r)=s-t,
\qquad
\Delta^F_\mu(x,r)=F(s)-F(t).
\]
Since $x\in\supp(\mu)$ and $\mu(\R^d)=\infty$, Lemma~\ref{lem:confinement} applies and
gives $s\in J_0$ and $t\in J_0$. In particular $s,t>0$, so $F(s)$ and $F(t)$ are
defined, and both points lie in the interval on which $F$ is assumed bi-Lipschitz.

By hypothesis $F$ is bi-Lipschitz on $J_0$ with constant $\lambda$, that is,
\[
\lambda^{-1}|s'-t'|\;\le\;|F(s')-F(t')|\;\le\;\lambda|s'-t'|
\qquad\text{for all } s',t'\in J_0.
\]
Taking $s'=s$ and $t'=t$, since $s,t\in J_0$, we get
\[
\lambda^{-1}|s-t|\;\le\;|F(s)-F(t)|\;\le\;\lambda|s-t|.
\]
Substituting $s-t=\Delta_\mu(x,r)$ and $F(s)-F(t)=\Delta^F_\mu(x,r)$, we obtain
\eqref{eq:comparable}.

Note that the lemma is purely pointwise: no integration or cancellation between scales is involved, and the
only property of $\mu$ that is used is that the two densities lie in $J_0$. If $s=t$
both sides of \eqref{eq:comparable} vanish, and if $s\ne t$ the inequality says that
the difference quotient $\bigl(F(s)-F(t)\bigr)/(s-t)$ lies in $[\lambda^{-1},\lambda]$
in absolute value.
\end{proof}

\begin{proof}[Proof of Theorem~\ref{thm:main}]
Let $B(x_0,R)$ be a ball with $x_0\in\supp(\mu)$. The measure $\mu$ assigns no mass to
$\R^d\setminus\supp(\mu)$, so in both \eqref{eq:carleson} and \eqref{eq:carlesonF} the
inner integral effectively ranges over $x\in B(x_0,R)\cap\supp(\mu)$, where
\eqref{eq:comparable} holds for every $r\in(0,R)$. Squaring \eqref{eq:comparable} and
integrating against $d\mu(x)\,\frac{dr}{r}$ over $B(x_0,R)\times(0,R)$ yields
\[
\lambda^{-2}\int_0^R\!\!\int_{B(x_0,R)}|\Delta_\mu|^2\,d\mu\,\frac{dr}{r}
\;\le\;
\int_0^R\!\!\int_{B(x_0,R)}|\Delta^F_\mu|^2\,d\mu\,\frac{dr}{r}
\;\le\;
\lambda^{2}\int_0^R\!\!\int_{B(x_0,R)}|\Delta_\mu|^2\,d\mu\,\frac{dr}{r}.
\]
This gives the two quantitative implications in the statement, and in particular
\eqref{eq:carleson} and \eqref{eq:carlesonF} are equivalent, with constants that differ
by at most a factor $\lambda^2$.

If $\mu$ is uniformly $n$-rectifiable, Theorem~\ref{thm:CGLT} gives \eqref{eq:carleson},
hence \eqref{eq:carlesonF}. Conversely, if \eqref{eq:carlesonF} holds, then
\eqref{eq:carleson} holds, and Theorem~\ref{thm:CGLT}, specifically Corollary~3.12
together with Theorem~3.11 of \cite{CGLT}, gives uniform $n$-rectifiability.
\end{proof}

\begin{proof}[Proof of Corollary~\ref{cor:log}, unbounded case]
The function $\log$ is $C^1$ on $(0,\infty)$ with derivative $1/t$. For $s,t\in
J_0=[c_0^{-1},c_0]$ the mean value theorem gives $\log s-\log t=(s-t)/\xi$ with $\xi$
between $s$ and $t$, hence $\xi\in J_0$ and $1/\xi\in[c_0^{-1},c_0]$. Thus $\log$ is
bi-Lipschitz on $J_0$ with $\lambda=c_0$, and Theorem~\ref{thm:main} applies. The
displayed identity follows from
$\log\Dm(x,r)=\log\mu(B(x,r))-n\log r$.
\end{proof}

\begin{remark}\label{rem:onesided}
The two implications in Theorem~\ref{thm:main} use the two halves of the bi-Lipschitz
condition separately. The proof of ``\eqref{eq:carleson} implies \eqref{eq:carlesonF}''
only uses the upper inequality $|F(s)-F(t)|\le\lambda|s-t|$ on $J_0$, that is, that
$F$ is Lipschitz on $J_0$, and the proof of ``\eqref{eq:carlesonF} implies
\eqref{eq:carleson}'' only uses the lower inequality
$|F(s)-F(t)|\ge\lambda^{-1}|s-t|$ on $J_0$, that is, that $F$ is injective on $J_0$
with $F^{-1}$ Lipschitz on $F(J_0)$. Hence, if $\mu$ is uniformly $n$-rectifiable and
$F$ is Lipschitz on $J_0$ then \eqref{eq:carlesonF} holds, and if \eqref{eq:carlesonF}
holds and $F^{-1}$ is Lipschitz on $F(J_0)$ then $\mu$ is uniformly $n$-rectifiable,
in both cases with constant $\lambda^2$. This is relevant in
Section~\ref{sec:degenerate}.
\end{remark}

\begin{remark}\label{rem:scope}
No monotonicity of $F$ is assumed in Theorem~\ref{thm:main}, and none is needed in
the proof, since only $|\Delta^F_\mu|$ enters \eqref{eq:carlesonF}. In fact
monotonicity is a consequence of the hypothesis: a bi-Lipschitz function on an interval
is continuous and injective, hence strictly monotone on $J_0$. Since $|\Delta^F_\mu|$
does not change when $F$ is replaced by $-F$, whether $F$ is increasing or decreasing
on $J_0$ is irrelevant. The Borel measurability of $F$ is needed only so that
$\Delta^F_\mu$ is a measurable function of $(x,r)$; it is automatic when $F$ is
continuous on $(0,\infty)$, and for the theorem itself only the values of $F$ on $J_0$
matter.

For $F\in C^1(0,\infty)$ the bi-Lipschitz condition on $J_0$ is equivalent to
\[
\lambda^{-1}\;\le\;|F'(t)|\;\le\;\lambda\qquad\text{for all } t\in J_0.
\]
Indeed, if $F$ is bi-Lipschitz on $J_0$ with constant $\lambda$, letting $s\to t$ in
the definition gives $\lambda^{-1}\le|F'(t)|\le\lambda$. Conversely, if
$\lambda^{-1}\le|F'|\le\lambda$ on $J_0$, then the mean value theorem gives
$|F(s)-F(t)|=|F'(\xi)|\,|s-t|$ for some $\xi$ between $s$ and $t$, which is the
bi-Lipschitz inequality. In other words, the only constraint on $F'$ is that it stays
away from $0$ and from $\infty$ on $J_0$; the condition $F'>0$ can always be arranged
by replacing $F$ with $-F$. In particular every $F\in C^1(0,\infty)$ with $F'\ne0$ on
$(0,\infty)$ satisfies the hypothesis, since $|F'|$ is then bounded above and below by
positive constants on the compact interval $J_0$. Examples are $\log t$, $t^p$ for
$p\ne 0$, $\arctan t$, and $e^{\pm t}$. The hypothesis fails when $F$ or $F^{-1}$ is
not Lipschitz on $J_0$, that is, when $F$ has a critical point or a point where it is
not Lipschitz inside $J_0$. This case is discussed in Section~\ref{sec:degenerate}.
\end{remark}

The quantity $\Delta^{\log}_\mu$ is also better suited than $\Delta_\mu$ for
numerical computation; we discuss this, and the reason why the regularizer in
UR--JEPA \cite{Le-URJEPA} is built from $\Delta^{\log}_\mu$, in
Appendix~\ref{app:numerics}.

\section{Smooth kernels}\label{sec:smooth}

In this section we prove Theorem~\ref{thm:smooth}. The argument is the same as in
Section~\ref{sec:proof}, once we know that $D_{\mu,\phi}(x,t)$ stays in the compact
interval $J_\phi$. Throughout this section $\phi$ is one of the two kernels of
Theorem~\ref{thm:smooth}. Note that in both cases $\phi$ is positive, radial, and
nonincreasing in $|x|$.

\begin{lemma}\label{lem:confinement-smooth}
Let $\mu$ be $n$-AD-regular with constant $c_0$ and $\mu(\R^d)=\infty$. Then for all
$x\in\supp(\mu)$ and all $t>0$,
\[
D_{\mu,\phi}(x,t)\in J_\phi,
\]
where $J_\phi$ is the interval in \eqref{eq:Jphi}.
\end{lemma}

\begin{proof}
As in the proof of Lemma~\ref{lem:confinement}, $\mu(\R^d)=\infty$ implies
$\diam(\supp(\mu))=\infty$, so that \eqref{eq:ADR} holds for all $r>0$.

For the lower bound, since $\phi$ is nonincreasing in $|x|$ we have
$\phi_t(y-x)\ge t^{-n}\phi(1)$ for $|y-x|<t$, and hence
\[
D_{\mu,\phi}(x,t)\;\ge\;\frac{\phi(1)}{t^n}\,\mu(B(x,t))\;\ge\;\phi(1)\,c_0^{-1}.
\]
For the upper bound we split $\R^d$ into $B(x,t)$ and the annuli
$A_k:=B(x,2^kt)\setminus B(x,2^{k-1}t)$, $k\ge1$. On $B(x,t)$ we have
$\phi_t(y-x)\le t^{-n}\phi(0)$, and on $A_k$ we have
$\phi_t(y-x)\le t^{-n}\phi(2^{k-1})$. Therefore, using \eqref{eq:ADR} at the radii
$2^kt$,
\[
D_{\mu,\phi}(x,t)
\;\le\;\frac{\phi(0)\,\mu(B(x,t))}{t^n}
+\sum_{k\ge1}\frac{\phi(2^{k-1})\,\mu(B(x,2^kt))}{t^n}
\;\le\;c_0\Bigl(\phi(0)+\sum_{k\ge1}2^{kn}\phi(2^{k-1})\Bigr)=C_\phi\,c_0.
\]
It remains to check that $C_\phi<\infty$. For $\phi(x)=e^{-|x|^{2N}}$ the terms
$2^{kn}e^{-2^{2N(k-1)}}$ decay faster than any geometric sequence. For
$\phi(x)=(1+|x|^2)^{-a}$ we have $\phi(2^{k-1})\le 2^{-2a(k-1)}$, so that
$2^{kn}\phi(2^{k-1})\le 2^{2a}\,2^{-(2a-n)k}$, which is summable since $a>n/2$.
\end{proof}

\begin{lemma}\label{lem:pointwise-smooth}
Under the hypotheses of Theorem~\ref{thm:smooth}, for all $x\in\supp(\mu)$ and $t>0$,
\begin{equation}\label{eq:comparable-smooth}
\lambda_\phi^{-1}\,|\Delta_{\mu,\phi}(x,t)|
\;\le\;|\Delta^F_{\mu,\phi}(x,t)|
\;\le\;\lambda_\phi\,|\Delta_{\mu,\phi}(x,t)|.
\end{equation}
If moreover $F\in C^1(0,\infty)$, then
\begin{equation}\label{eq:comparable-smooth-tilde}
\lambda_\phi^{-1}\,|\widetilde\Delta_{\mu,\phi}(x,t)|
\;\le\;|\widetilde\Delta^F_{\mu,\phi}(x,t)|
\;\le\;\lambda_\phi\,|\widetilde\Delta_{\mu,\phi}(x,t)|.
\end{equation}
\end{lemma}

\begin{proof}
Write $s=D_{\mu,\phi}(x,t)$ and $s'=D_{\mu,\phi}(x,2t)$. By
Lemma~\ref{lem:confinement-smooth}, $s,s'\in J_\phi$, and
\eqref{eq:comparable-smooth} is the bi-Lipschitz property of $F$ on $J_\phi$ applied
to $s$ and $s'$, exactly as in Lemma~\ref{lem:pointwise}.

For \eqref{eq:comparable-smooth-tilde}, note first that $t\mapsto D_{\mu,\phi}(x,t)$
is $C^1$ on $(0,\infty)$, since $\phi_t$ is smooth in $t$ and the differentiation
under the integral sign is justified by the decay of $\phi$ and of $\partial_t\phi_t$
together with \eqref{eq:ADR}; this is used in \cite{CGLT} to define
$\widetilde\Delta_{\mu,\phi}$. By the chain rule,
\[
\widetilde\Delta^F_{\mu,\phi}(x,t)
=t\,\partial_t F\bigl(D_{\mu,\phi}(x,t)\bigr)
=F'\bigl(D_{\mu,\phi}(x,t)\bigr)\,\widetilde\Delta_{\mu,\phi}(x,t).
\]
Since $F\in C^1(0,\infty)$ is bi-Lipschitz on $J_\phi$ with constant
$\lambda_\phi$, letting $s'\to s$ in the bi-Lipschitz inequality gives
$\lambda_\phi^{-1}\le|F'(s)|\le\lambda_\phi$ for all $s\in J_\phi$
(see Remark~\ref{rem:scope}). Applying
this with $s=D_{\mu,\phi}(x,t)\in J_\phi$ gives
\eqref{eq:comparable-smooth-tilde}.
\end{proof}

\begin{proof}[Proof of Theorem~\ref{thm:smooth}]
Let $B(x_0,R)$ be a ball with $x_0\in\supp(\mu)$. Since $\mu$ gives no mass to
$\R^d\setminus\supp(\mu)$, we may integrate \eqref{eq:comparable-smooth} squared
against $d\mu(x)\,\frac{dt}{t}$ over $B(x_0,R)\times(0,R)$, and we obtain that
\eqref{eq:carlesonFphi} holds if and only if the Carleson condition for
$\Delta_{\mu,\phi}$ holds, with constants that differ by at most a factor
$\lambda_\phi^2$. By \cite[Theorem 1.2]{CGLT}, the latter is equivalent to the
uniform $n$-rectifiability of $\mu$. This proves that (a) and (b) are equivalent. If
$F\in C^1(0,\infty)$, the same argument with \eqref{eq:comparable-smooth-tilde} in
place of \eqref{eq:comparable-smooth} shows that \eqref{eq:carlesonFphitilde} is
equivalent to the Carleson condition for $\widetilde\Delta_{\mu,\phi}$, which by
\cite[Theorem 1.2]{CGLT} is again equivalent to (a).

For $F=\log$, the mean value theorem on $J_\phi$ as in the proof of
Corollary~\ref{cor:log} gives that $\log$ is bi-Lipschitz on $J_\phi$ with constant
$\max\bigl(C_\phi c_0,\,c_0/\phi(1)\bigr)$, which is the largest of the
values of $1/s$ and $s$ for $s\in J_\phi$.
\end{proof}

\subsection*{Higher order square functions}

In \cite[Proposition 1.3]{CGLT} the characterization is also obtained with the
higher order square functions $\Delta^k_{\mu,\phi}$ and $\widetilde\Delta^k_{\mu,\phi}$,
$k\ge1$, where $\Delta^k_{\mu,\phi}(x,t)$ is the $k$-th iterated difference of
$s\mapsto D_{\mu,\phi}(x,s)$ over the scales $t,2t,\dots,2^kt$ and
\[
\widetilde\Delta^k_{\mu,\phi}(x,t)=\int\partial^k_\phi(y-x,t)\,d\mu(y)
=t^k\partial_t^kD_{\mu,\phi}(x,t),
\qquad \partial^k_\phi(x,t)=t^k\partial_t^k\phi_t(x).
\]
For $F\in C^k(0,\infty)$ we consider the transformed quantities
\[
\widetilde\Delta^{F,k}_{\mu,\phi}(x,t):=t^k\partial_t^k\,F\bigl(D_{\mu,\phi}(x,t)\bigr),
\qquad k\ge1,
\]
so that $\widetilde\Delta^{F,1}_{\mu,\phi}=\widetilde\Delta^F_{\mu,\phi}$. If $F$ is
affine on $J_\phi$, say $F(s)=as+b$ with $a\ne0$, then
$\widetilde\Delta^{F,k}_{\mu,\phi}=a\widetilde\Delta^k_{\mu,\phi}$ and there is nothing to
prove. For nonlinear $F$ the change of variables argument of Lemma~\ref{lem:pointwise}
breaks down for $k\ge2$: for $k=2$ the chain rule gives
\begin{equation}\label{eq:k2}
\widetilde\Delta^{F,2}_{\mu,\phi}
=F'\bigl(D_{\mu,\phi}\bigr)\,\widetilde\Delta^2_{\mu,\phi}
+F''\bigl(D_{\mu,\phi}\bigr)\,\bigl(\widetilde\Delta^1_{\mu,\phi}\bigr)^2,
\end{equation}
and in general,
\begin{equation}\label{eq:faa}
\widetilde\Delta^{F,k}_{\mu,\phi}
=F'\bigl(D_{\mu,\phi}\bigr)\,\widetilde\Delta^k_{\mu,\phi}
+\sum F^{(m)}\bigl(D_{\mu,\phi}\bigr)\,c_{m_1,\dots,m_k}
\prod_{j=1}^{k}\bigl(\widetilde\Delta^{j}_{\mu,\phi}\bigr)^{m_j},
\end{equation}
where the sum runs over the $(m_1,\dots,m_k)$ with $\sum_jjm_j=k$ and
$m=\sum_jm_j\ge2$, and the $c_{m_1,\dots,m_k}$ are combinatorial constants. The extra
terms are products of lower order derivatives of $D_{\mu,\phi}$ and are not controlled
pointwise by $\widetilde\Delta^k_{\mu,\phi}$, so $\widetilde\Delta^{F,k}_{\mu,\phi}$ and
$\widetilde\Delta^k_{\mu,\phi}$ are not comparable pointwise in general. Nevertheless,
if the derivatives of $F$ up to order $k$ are bounded on $J_\phi$, the characterization
still holds. The two ingredients of the proof are that the $\widetilde\Delta^j_{\mu,\phi}$ are
bounded, and that $\widetilde\Delta^{F,j-1}_{\mu,\phi}$ can be recovered from
$\widetilde\Delta^{F,j}_{\mu,\phi}$ by an integration in the scale variable.

\begin{proposition}\label{prop:higher}
Let $\mu$ be $n$-AD-regular with constant $c_0$ and $\mu(\R^d)=\infty$, let $\phi$ be
as in Theorem~\ref{thm:smooth}, and let $k\ge1$. Let $F\in C^k(0,\infty)$ satisfy
\[
\lambda^{-1}\le|F'(s)|\le\lambda
\quad\text{and}\quad
|F^{(m)}(s)|\le\Lambda\ \ (2\le m\le k)
\qquad\text{for all } s\in J_\phi.
\]
Then $\mu$ is uniformly $n$-rectifiable if and only if there exists a constant $c$
such that, for any ball $B(x_0,R)$ centered at $\supp(\mu)$,
\begin{equation}\label{eq:carlesonFk}
\int_0^R\!\!\int_{x\in B(x_0,R)}|\widetilde\Delta^{F,k}_{\mu,\phi}(x,t)|^2\,d\mu(x)\,\frac{dt}{t}
\;\le\;c\,R^n.
\end{equation}
The constants depend on $k$, $\lambda$, $\Lambda$, $\phi$, $n$ and $c_0$.
\end{proposition}

The proof, together with the uniform bound on the $\widetilde\Delta^j_{\mu,\phi}$
(Lemma~\ref{lem:bounded-higher}), is given in Appendix~\ref{app:higher}.

\begin{remark}
The hypotheses of Proposition~\ref{prop:higher} are satisfied by $F=\log$ for every
$k$, with $\Lambda$ depending on $k$ and $J_\phi$, since $(\log)^{(m)}(s)=(-1)^{m-1}(m-1)!\,s^{-m}$
is bounded on $J_\phi$. For the iterated differences $\Delta^k_{\mu,\phi}$ an analogue of the necessity
part can be obtained by a similar, though not identical, argument: expanding $F$ by
Taylor's formula around $D_{\mu,\phi}(x,2^kt)$, the $k$-th difference of
$F(D_{\mu,\phi})$ equals $F'(D_{\mu,\phi}(x,2^kt))\Delta^k_{\mu,\phi}(x,t)$ plus a sum
of products of first order differences $\Delta_{\mu,\phi}(x,2^it)$, $0\le i<k$, which
are bounded, and the Carleson estimate for these shifted scales is obtained from
\eqref{eq:carleson} on the enlarged ball $B(x_0,2^kR)$. We omit the details. The
sufficiency part does not follow in the same way, because the relation
$\Delta^{k-1}_{\mu,\phi}(x,t)-\Delta^{k-1}_{\mu,\phi}(x,2t)=\Delta^k_{\mu,\phi}(x,t)$
has no decaying weight and $\Delta^{k-1}_{\mu,\phi}(x,2^It)$ need not tend to $0$ as
$I\to\infty$. We do not pursue this here.
\end{remark}

\section{Bounded support}\label{sec:bounded}

In this section we assume that $\mu(\R^d)<\infty$, and we set
$R_0:=\diam(\supp(\mu))<\infty$. Recall that $J_0=[c_0^{-1},c_0]$ is the interval
of Lemma~\ref{lem:confinement}. Since \eqref{eq:ADR} only applies at radii
$r\le R_0$, the density $\Dm(x,2r)$ can leave $J_0$ when $2r>R_0$, and we will work
instead on the larger interval
\[
J_1:=\bigl[2^{-n}c_0^{-1},\,c_1\bigr],\qquad c_1:=5^d\,c_0,
\]
which contains $J_0$. The constant $c_1$ comes from the bound
$c_0^{-1}R_0^n\le\mu(\R^d)\le c_1R_0^n$, proved by a covering argument in
Lemma~\ref{lem:totalmass} in Appendix~\ref{app:bounded}. 

We now distinguish three ranges of scales. The density lies in $J_0$ only in the first
one.
\begin{itemize}
\item[(i)] $0<r\le R_0/2$: both $r$ and $2r$ are admissible in \eqref{eq:ADR}, and
$\Dm(x,r),\Dm(x,2r)\in J_0$.
\item[(ii)] $R_0/2<r\le R_0$: the radius $2r$ exceeds $R_0$, so
$\supp(\mu)\subset B(x,2r)$ and $\mu(B(x,2r))=\mu(\R^d)$. By
Lemma~\ref{lem:totalmass} and $R_0/2<r\le R_0$,
$\Dm(x,2r)=\mu(\R^d)(2r)^{-n}\in\bigl[2^{-n}c_0^{-1},\,c_1\bigr]=J_1$.
Also $\Dm(x,r)\in J_0\subset J_1$, since $c_0\le c_1$.
\item[(iii)] $r>R_0$: both balls contain $\supp(\mu)$, so
$\Dm(x,r)=\mu(\R^d)r^{-n}$ and $\Dm(x,2r)=\mu(\R^d)(2r)^{-n}$ tend to $0$ as
$r\to\infty$, and the behavior of $F$ near $0$ enters.
\end{itemize}

\begin{proposition}\label{prop:bounded}
Let $\mu$ be $n$-AD-regular with constant $c_0$ and $\mu(\R^d)<\infty$, and let
$F\colon(0,\infty)\to\R$ be a Borel function and bi-Lipschitz on
$J_1=[2^{-n}c_0^{-1},c_1]$ with constant $\lambda_1$, where $c_1=5^dc_0$.
\begin{itemize}
\item[(a)] If \eqref{eq:carlesonF} holds for some constant $c$, then $\mu$ is uniformly
$n$-rectifiable. No hypothesis on $F$ near $0$ is required.
\item[(b)] Conversely, if $\mu$ is uniformly $n$-rectifiable and in addition
\begin{equation}\label{eq:tail}
\sup_{0<\varepsilon\le c_1}\ \varepsilon\int_\varepsilon^{c_1}
\bigl|F(t)-F(2^{-n}t)\bigr|^2\,\frac{dt}{t}\;<\;\infty,
\end{equation}
then \eqref{eq:carlesonF} holds. Condition \eqref{eq:tail} is satisfied whenever
$t\mapsto F(t)-F(2^{-n}t)$ is bounded on $(0,c_1]$, and for $F$ continuous on
$(0,\infty)$ this is the same as being bounded near $0$; in particular \eqref{eq:tail}
holds for $F=\log$, for which $F(t)-F(2^{-n}t)\equiv n\log 2$.
\end{itemize}
Consequently, for $F=\log$ the full equivalence of Corollary~\ref{cor:log} holds with no
restriction on the total mass, with $\lambda_1=5^dc_0$.
\end{proposition}

The proof is given in Appendix~\ref{app:bounded}.

\begin{remark}
Some hypothesis on $F$ near $0$ cannot be dispensed with in (b). In (a)
the large scales of $\Delta^F_\mu$ are not used at all, and only the tail of
$\Delta_\mu$ has to be estimated, so the behavior of $F$ near $0$ does not enter. In
(b) one has to estimate the large scales of $\Delta^F_\mu$ itself. Take for instance
$F(t)=-t^{-\gamma}$ with $\gamma>0$, which is strictly increasing and smooth on
$(0,\infty)$. Then $F(t)-F(2^{-n}t)=(2^{n\gamma}-1)\,t^{-\gamma}$, and
\[
\varepsilon\int_\varepsilon^{c_1}\bigl|F(t)-F(2^{-n}t)\bigr|^2\,\frac{dt}{t}
=\frac{(2^{n\gamma}-1)^2}{2\gamma}\,\varepsilon\bigl(\varepsilon^{-2\gamma}-c_1^{-2\gamma}\bigr)
\approx\varepsilon^{1-2\gamma}\qquad\text{as }\varepsilon\to0.
\]
Thus \eqref{eq:tail} holds for $0<\gamma\le\frac12$, although $F(t)-F(2^{-n}t)$ is
unbounded near $0$, and fails for $\gamma>\frac12$. For $\gamma>\frac12$ the identity
in the proof of (b) (see Appendix~\ref{app:bounded}) shows that the large-scale part of \eqref{eq:carlesonF} is
comparable to $\mu(\R^d)\,\varepsilon^{-2\gamma}$ with $\varepsilon=\mu(\R^d)R^{-n}$,
that is, to $\mu(\R^d)^{1-2\gamma}R^{2\gamma n}$, which is not bounded by $cR^n$ as
$R\to\infty$. Hence \eqref{eq:carlesonF} fails at large scales for every uniformly
$n$-rectifiable $\mu$ with $\mu(\R^d)<\infty$, and for such $F$ the square function
condition is strictly stronger than uniform rectifiability. This does not happen when
$\mu(\R^d)=\infty$.
\end{remark}

\section{Degenerate transformations}\label{sec:degenerate}

In this section we discuss the case where $F$ is strictly increasing and continuous
but not bi-Lipschitz on $J_0$. Examples are $F(t)=(t-c_*)^3$ with
$c_*\in\operatorname{int}J_0$, which has a critical point inside $J_0$, and functions
$F$ which are H\"older but not Lipschitz. Throughout this section we assume that
$\mu(\R^d)=\infty$. The two implications in Theorem~\ref{thm:main} behave
differently, and we state them separately with different hypotheses.

\begin{proposition}\label{prop:degenerate-easy}
Let $F$ be strictly increasing and $\alpha$-H\"older on $J_0$ for some
$\alpha\in(0,1]$, with constant $H$. If $\mu$ is uniformly $n$-rectifiable, then for
every ball $B(x_0,R)$ centered at $\supp(\mu)$,
\[
\int_0^R\!\!\int_{B(x_0,R)}|\Delta^F_\mu(x,r)|^{2/\alpha}\,d\mu(x)\,\frac{dr}{r}
\;\le\;H^{2/\alpha}\,c'\,R^n,
\]
where $c'$ is the constant in \eqref{eq:carleson}.
\end{proposition}

\begin{proof}
By Lemma~\ref{lem:confinement} we have $|\Delta^F_\mu|\le H|\Delta_\mu|^\alpha$
pointwise on $\supp(\mu)\times(0,\infty)$, and the claim follows from
\eqref{eq:carleson}.
\end{proof}

Note that this argument does not give a bound for the square function of
$\Delta^F_\mu$ with exponent $2$ when $\alpha<1$: the pointwise inequality
$|\Delta^F_\mu|^2\le H^2|\Delta_\mu|^{2\alpha}$ only controls $|\Delta^F_\mu|^2$ by a
power of $|\Delta_\mu|$ smaller than $2$, and \eqref{eq:carleson} says nothing about
$\int\!\!\int|\Delta_\mu|^{2\alpha}\,d\mu\,\frac{dr}{r}$. We do not know whether
\eqref{eq:carlesonF} can actually fail for a uniformly $n$-rectifiable $\mu$ and a
H\"older $F$ which is not Lipschitz on $J_0$.

For the converse implication, the argument in \cite[Section 3]{CGLT} is qualitative
and should still apply, but the reduction used in Section~\ref{sec:proof} does not. The
passage from the non-smooth condition \eqref{eq:carleson} to its smooth version in
\cite[Corollary 3.12]{CGLT} is done by taking convex combinations, which is linear in
the density, and this does not commute with a nonlinear $F$. We state the expected
result as a conjecture, together with the strategy we expect to prove it.

\begin{conjecture}\label{prop:degenerate-hard}
Let $F$ be strictly increasing and continuous on $J_0$, and let $\mu$ be
$n$-AD-regular with $\mu(\R^d)=\infty$. If \eqref{eq:carlesonF} holds, then $\mu$ is
uniformly $n$-rectifiable.
\end{conjecture}

The strategy is to go through the weak constant density condition and
\cite[Theorem 2.3]{CGLT}, as in \cite[Section 3]{CGLT}: the smallness of the local
$\Delta^F$-square function should force, via the uniform continuity of $F^{-1}$ on the
compact interval $F(J_0)$, the smallness of the local oscillation of $\Dm$ itself, with
a modulus depending on $F$; Chebyshev's inequality then makes the complement a
Carleson set, which is the weak constant density condition, and
\cite[Theorem 2.3]{CGLT} gives uniform rectifiability. This replaces the compactness
argument of \cite[Lemmas 3.1--3.6]{CGLT} at one point: where those lemmas conclude
$|\mu(B(y,t))-c\,t^n|<\varepsilon r^n$ from smallness of the $\widetilde\Delta$-square
function, one would start from smallness of the $\Delta^F$-integrals, obtain that
$F(\Dm)$ is nearly constant along dyadic chains of scales, and invert $F$; since
$F^{-1}$ is uniformly continuous on $F(J_0)$, near-constancy of $F(\Dm)$ gives
near-constancy of $\Dm$ with a modulus $\omega_{F^{-1}}$. The quantitative bookkeeping in
\cite[Lemma 3.5]{CGLT} (the exponent $\delta^{n+4}$, the Cauchy--Schwarz step) has to
be redone with $\omega_{F^{-1}}$ in place of a linear modulus, and for $F$ with a very
degenerate modulus the exponents in the Chebyshev step may have to be adjusted. We
have not carried out these details. For bi-Lipschitz $F$ none of this is needed, since
Section~\ref{sec:proof} applies.

\begin{remark}
To summarize, for $F$ which is only H\"older the arguments of this paper give the
following. The sufficiency of \eqref{eq:carlesonF} for uniform rectifiability is
expected to survive, but we have not proved it; this is
Conjecture~\ref{prop:degenerate-hard}. The necessity of \eqref{eq:carlesonF} is no
longer obtained by the change of variables argument, which only gives the
$L^{2/\alpha}$ estimate of Proposition~\ref{prop:degenerate-easy}, and we do not know
whether it holds. Thus the bi-Lipschitz functions of Theorem~\ref{thm:main} are the
class for which we prove the equivalence with quantitative constants; we do not claim
that the equivalence fails outside this class.
\end{remark}

\section{$n$-rectifiability without AD-regularity}\label{sec:qualitative}

In this section we drop the AD-regularity assumption and consider the qualitative
notion of rectifiability. Following \cite[Definition 16.6]{Ma}, a Radon measure $\mu$
in $\R^d$ is \emph{$n$-rectifiable} if $\mu$ vanishes outside an $n$-rectifiable set
$E\subset\R^d$ and $\mu$ is absolutely continuous with respect to $\mathcal H^n|_E$;
here a set $E$ is $n$-rectifiable if $\mathcal H^n(E\setminus\bigcup_i f_i(\R^n))=0$ for
some Lipschitz maps $f_i\colon\R^n\to\R^d$. For a Radon measure $\mu$ and
$x\in\R^d$ we write
\[
\Theta^{n,*}(x,\mu):=\limsup_{r\to0}\frac{\mu(B(x,r))}{r^n},
\qquad
\Theta^{n}_*(x,\mu):=\liminf_{r\to0}\frac{\mu(B(x,r))}{r^n}
\]
for the upper and lower $n$-dimensional densities of $\mu$ at $x$. The square function
version of Preiss' theorem \cite{Pr} proved by Tolsa and Toro reads as follows.

\begin{theorem}[{\cite[Theorem 1.1]{TT}}]\label{thm:TT}
Let $\mu$ be a Radon measure in $\R^d$ such that
$0<\Theta^n_*(x,\mu)\le\Theta^{n,*}(x,\mu)<\infty$ for $\mu$-a.e.\ $x\in\R^d$. The
following are equivalent:
\begin{itemize}
\item[(a)] $\mu$ is $n$-rectifiable.
\item[(b)] $\displaystyle\int_0^1|\Delta_\mu(x,r)|^2\,\frac{dr}{r}<\infty$ for
$\mu$-a.e.\ $x\in\R^d$.
\item[(c)] $\displaystyle\lim_{r\to0}\Delta_\mu(x,r)=0$ for $\mu$-a.e.\ $x\in\R^d$.
\end{itemize}
\end{theorem}

Note that no AD-regularity and no doubling condition is assumed, and that the density
bounds are only asymptotic and pointwise, with constants depending on $x$. The
change of variables of Section~\ref{sec:proof} still applies, because at $\mu$-a.e.\
point the density stays in a compact interval of $(0,\infty)$ for all the scales
involved. We say that $F\colon(0,\infty)\to\R$ is \emph{locally bi-Lipschitz} if it is
bi-Lipschitz on every compact interval $J\subset(0,\infty)$, with a constant $\lambda_J$
depending on $J$; by Remark~\ref{rem:scope}, every $F\in C^1(0,\infty)$ with $F'\ne0$
is locally bi-Lipschitz, and such an $F$ is automatically continuous and strictly
monotone.

\begin{lemma}\label{lem:confinement-ae}
Let $\mu$ be a Radon measure in $\R^d$ and let $x\in\R^d$ be such that
$0<\Theta^n_*(x,\mu)\le\Theta^{n,*}(x,\mu)<\infty$. Then there exist
$0<a_x\le b_x<\infty$ such that
\[
a_x\;\le\;\Dm(x,r)\;\le\;b_x\qquad\text{for all } 0<r\le 2.
\]
In particular $\Dm(x,r),\Dm(x,2r)\in J_x:=[a_x,b_x]$ for all $0<r\le1$.
\end{lemma}

\begin{proof}
Write $\theta_*=\Theta^n_*(x,\mu)$ and $\theta^*=\Theta^{n,*}(x,\mu)$. By the definition
of $\liminf$ and $\limsup$ there is $r_x\in(0,1]$ such that
\[
\tfrac12\theta_*\;\le\;\Dm(x,r)\;\le\;2\theta^*\qquad\text{for } 0<r\le r_x.
\]
For $r_x\le r\le2$ we use that $r\mapsto\mu(B(x,r))$ is nondecreasing: on the one
hand $\mu(B(x,r))\ge\mu(B(x,r_x))\ge\tfrac12\theta_*r_x^n$, so that
$\Dm(x,r)\ge\tfrac12\theta_*r_x^n\,2^{-n}$; on the other hand
$\mu(B(x,r))\le\mu(B(x,2))<\infty$, since $\mu$ is Radon, so that
$\Dm(x,r)\le\mu(B(x,2))\,r_x^{-n}$. Hence the claim holds with
\[
a_x:=2^{-n-1}\theta_*r_x^n\;\le\;\tfrac12\theta_*,
\qquad
b_x:=\max\bigl(2\theta^*,\;\mu(B(x,2))\,r_x^{-n}\bigr).
\qedhere
\]
\end{proof}

\begin{theorem}\label{thm:TT-F}
Let $\mu$ be a Radon measure in $\R^d$ such that
$0<\Theta^n_*(x,\mu)\le\Theta^{n,*}(x,\mu)<\infty$ for $\mu$-a.e.\ $x\in\R^d$, and
let $F\colon(0,\infty)\to\R$ be locally bi-Lipschitz. The following are equivalent:
\begin{itemize}
\item[(a)] $\mu$ is $n$-rectifiable.
\item[(b)] $\displaystyle\int_0^1|\Delta^F_\mu(x,r)|^2\,\frac{dr}{r}<\infty$ for
$\mu$-a.e.\ $x\in\R^d$.
\item[(c)] $\displaystyle\lim_{r\to0}\Delta^F_\mu(x,r)=0$ for $\mu$-a.e.\ $x\in\R^d$.
\end{itemize}
\end{theorem}

\begin{proof}
Let $x$ be a point where $0<\Theta^n_*(x,\mu)\le\Theta^{n,*}(x,\mu)<\infty$; then
$x\in\supp(\mu)$, since $\mu(B(x,r))>0$ for small $r$. Let
$J_x=[a_x,b_x]$ be the interval of Lemma~\ref{lem:confinement-ae} and let
$\lambda_x:=\lambda_{J_x}$ be the bi-Lipschitz constant of $F$ on $J_x$. For
$0<r\le1$ both $\Dm(x,r)$ and $\Dm(x,2r)$ belong to $J_x$, and the argument of
Lemma~\ref{lem:pointwise} gives
\begin{equation}\label{eq:comparable-ae}
\lambda_x^{-1}\,|\Delta_\mu(x,r)|\;\le\;|\Delta^F_\mu(x,r)|\;\le\;\lambda_x\,|\Delta_\mu(x,r)|
\qquad\text{for all } 0<r\le1.
\end{equation}
The constant $\lambda_x$ depends on $x$ but not on $r$. Therefore, at every such $x$,
$\int_0^1|\Delta^F_\mu(x,r)|^2\,\frac{dr}{r}<\infty$ if and only if
$\int_0^1|\Delta_\mu(x,r)|^2\,\frac{dr}{r}<\infty$, and
$\lim_{r\to0}\Delta^F_\mu(x,r)=0$ if and only if $\lim_{r\to0}\Delta_\mu(x,r)=0$. Thus
(b) and (c) are equivalent to conditions (b) and (c) of Theorem~\ref{thm:TT}
respectively, and the theorem follows from Theorem~\ref{thm:TT}.
\end{proof}

\begin{corollary}\label{cor:TT-log}
Let $\mu$ be a Radon measure in $\R^d$ such that
$0<\Theta^n_*(x,\mu)\le\Theta^{n,*}(x,\mu)<\infty$ for $\mu$-a.e.\ $x\in\R^d$. The
following are equivalent:
\begin{itemize}
\item[(a)] $\mu$ is $n$-rectifiable.
\item[(b)] $\displaystyle\int_0^1\Bigl|\log\frac{\mu(B(x,r))}{\mu(B(x,2r))}+n\log2\Bigr|^2
\,\frac{dr}{r}<\infty$ for $\mu$-a.e.\ $x\in\R^d$.
\item[(c)] $\displaystyle\lim_{r\to0}\frac{\mu(B(x,r))}{\mu(B(x,2r))}=2^{-n}$ for
$\mu$-a.e.\ $x\in\R^d$.
\end{itemize}
The same holds for a Borel set $E\subset\R^d$ with $\mathcal H^n(E)<\infty$ and
$\Theta^n_*(x,\mathcal H^n|_E)>0$ for $\mathcal H^n$-a.e.\ $x\in E$, with
$\mu=\mathcal H^n|_E$ and ``$E$ is $n$-rectifiable'' in place of (a).
\end{corollary}

\begin{proof}
The function $\log$ is $C^1$ on $(0,\infty)$ with $(\log)'(t)=1/t\ne0$, so it is
locally bi-Lipschitz, and $\Delta^{\log}_\mu(x,r)=\log\frac{\mu(B(x,r))}{\mu(B(x,2r))}+n\log2$
by Corollary~\ref{cor:log}. Condition (c) of Theorem~\ref{thm:TT-F} for $F=\log$ says
that $\log\frac{\mu(B(x,r))}{\mu(B(x,2r))}\to-n\log2$, which is (c) above by continuity
of $\exp$. For the set version, recall that $\Theta^{n,*}(x,\mathcal H^n|_E)\le2^n$ for
$\mathcal H^n$-a.e.\ $x\in E$ when $\mathcal H^n(E)<\infty$ (see
\cite[Theorem 6.2]{Ma}, where the density is normalized by $(2r)^n$ instead of
$r^n$), so that $\mu=\mathcal H^n|_E$ satisfies the hypotheses, and
that $\mathcal H^n|_E$ is $n$-rectifiable as a measure if and only if $E$ is
$n$-rectifiable as a set, which is immediate from the definitions. The set version
is stated in \cite[Corollary 1.2]{TT}.
\end{proof}

\begin{remark}
Condition (c) of Corollary~\ref{cor:TT-log} is a statement about the doubling ratio
alone: at $\mu$-a.e.\ point the ratio $\mu(B(x,r))/\mu(B(x,2r))$ has the limit $2^{-n}$
that an $n$-plane has. Note that this is a condition on the existence of a limit of a
ratio of masses, in the spirit of Preiss' theorem, and that it involves the dimension
$n$ only through the value of the limit. Let us also mention that for the equivalence
of (c) in Theorems~\ref{thm:TT} and~\ref{thm:TT-F} much less than bi-Lipschitz is
needed: it suffices that $F$ be continuous and strictly monotone on $(0,\infty)$, since
then $F$ and $F^{-1}$ are uniformly continuous on the compact intervals $J_x$ and
$F(J_x)$ respectively. The bi-Lipschitz condition is used only for the equivalence of
the square function conditions (b).

We stress that Theorem~\ref{thm:TT-F} is qualitative. The comparability constant
$\lambda_x$ in \eqref{eq:comparable-ae} depends on $x$ through the interval $J_x$,
which in turn depends on the densities at $x$ and on the scale $r_x$ at which they
stabilize, and nothing is claimed about the size of the integrals in (b). The
quantitative statement, with Carleson constants, is Theorem~\ref{thm:main}, and it is
there that AD-regularity is used.
\end{remark}

\section{Concluding remarks}

We have shown that the square function characterization of uniform rectifiability in
\cite{CGLT} is invariant under bi-Lipschitz changes of variables in the density, both
for the density $\mu(B(x,r))/r^n$ and for its smooth versions $\phi_t*\mu(x)$, with
constants that differ by at most a factor $\lambda^2$, and that for $F=\log$ the
statement holds with no restriction on $\mu(\R^d)$. We have also shown that the
qualitative characterization of $n$-rectifiability of \cite{TT} is invariant under
locally bi-Lipschitz changes of variables, without any AD-regularity or doubling
assumption. We end with two questions. First,
for measures which are not AD-regular the choice $F=\log$ is no longer a change of
variables at the quantitative level: $\Delta^{\log}_\mu$ depends only on the doubling
ratio and is finite and scale-invariant for any doubling measure, while $\Delta_\mu$
need not be comparable to it with uniform constants. It is not known to us whether a
Carleson-type condition on $\Delta^{\log}_\mu$ characterizes some quantitative form of
rectifiability for doubling measures which are not AD-regular, in the
spirit of \cite{To-memoir} and \cite{ADT}, and the arguments in this paper do not apply
to this case; Section~\ref{sec:qualitative} gives only the qualitative statement. This question is relevant to \cite{Le-URJEPA}, where $\Delta^{\log}_\mu$
is used as a regularizer without imposing AD-regularity: in that setting
Theorems~\ref{thm:main} and~\ref{thm:smooth} do not apply as stated, and a positive
answer would give the regularizer a characterization of rectifiability to rest on. Second, a proof of
Conjecture~\ref{prop:degenerate-hard} along the lines sketched in
Section~\ref{sec:degenerate} would give the sufficiency implication for all
continuous strictly monotone $F$, if the strategy can be carried out. The difficulty
seems technical, but the constants would depend on the modulus of continuity of
$F^{-1}$, and a precise statement would have to keep track of this dependence.

\appendix

\section{Numerical evaluation}\label{app:numerics}

\begin{remark}[numerical evaluation]\label{rem:numerics}
Although Corollary~\ref{cor:log} is only a change of variables, the quantity
$\Delta^{\log}_\mu$ is better suited than $\Delta_\mu$ for numerical computation. This
is the reason why the regularizer in UR--JEPA \cite{Le-URJEPA} is built from
$\Delta^{\log}_\mu$, where $\mu$ is the empirical measure of the learned embeddings,
rather than from $\Delta_\mu$. We list the following three reasons.

(i) \emph{Only a ratio of masses is needed.} Since $B(x,r)\subset B(x,2r)$ we have
$\mu(B(x,r))\le\mu(B(x,2r))$, and by \eqref{eq:ADR}, for $x\in\supp(\mu)$ and
$2r\le\diam(\supp(\mu))$,
\begin{equation}\label{eq:ratio}
2^{-n}c_0^{-2}\;\le\;\frac{\mu(B(x,r))}{\mu(B(x,2r))}\;\le\;1,
\qquad\text{hence}\qquad
-2\log c_0\;\le\;\Delta^{\log}_\mu(x,r)\;\le\;n\log 2.
\end{equation}
Thus $\Delta^{\log}_\mu(x,r)$ is obtained by evaluating $\log$ at a number in the fixed
interval $[2^{-n}c_0^{-2},1]$, whatever the scale $r$ is, and no quantity of size
$r^{n}$ or $r^{-n}$ is ever formed. To compute $\Delta_\mu(x,r)$ one has to form
$\mu(B(x,r))/r^n$ and $\mu(B(x,2r))/(2r)^n$ separately; when $r$ ranges over many dyadic
scales these are quotients of numbers whose magnitudes depend on the units of length and
of mass, and in floating point arithmetic $r^{\pm n}$ overflows or underflows once
$n\log_2 r$ exceeds the exponent range. In practice one should evaluate the ratio
first and take one logarithm, that is, compute
$\log\bigl(\mu(B(x,r))/\mu(B(x,2r))\bigr)$ rather than
$\log\mu(B(x,r))-\log\mu(B(x,2r))$, so that the result is the logarithm of a number
in $[2^{-n}c_0^{-2},1]$ and no cancellation between two large logarithms occurs.

(ii) \emph{Invariance under normalization.} For $a>0$ we have
$\Delta^{\log}_{a\mu}=\Delta^{\log}_\mu$, while $\Delta_{a\mu}=a\Delta_\mu$. Taking
into account that the measure $d\mu$ in the integrals is also replaced by $a\,d\mu$,
the Carleson constant in \eqref{eq:carleson} scales like $a^3$, while the one in
\eqref{eq:carlesonF} for $F=\log$ scales like $a$, the factor coming only from the
measure. Similarly, if $\mu$ is replaced by
its image under the dilation $x\mapsto sx$ and $r$ by $sr$, then $\Delta^{\log}_\mu$
is unchanged, whereas $\Delta_\mu$ is multiplied by $s^{-n}$ unless $\mu$ is rescaled
accordingly. In particular, for an empirical measure $\mu_N=\frac1N\sum_{i=1}^N
\delta_{y_i}$ the quantity $\Delta^{\log}_{\mu_N}(x,r)$ is the logarithm of the ratio of
two point counts, plus $n\log 2$, and depends neither on $N$ nor on the choice of
units.

(iii) \emph{Dependence on the dimension.} The dimension $n$ enters
$\Delta^{\log}_\mu$ only through the additive constant $n\log 2$. If $n$ is replaced
by $n+\delta$, then $\Delta^{\log}_\mu$ changes by the constant $\delta\log 2$ at every
$(x,r)$, whereas $\Delta_\mu$ acquires the factors $r^{-\delta}$ and $(2r)^{-\delta}$,
which are not bounded uniformly in $r$. This is relevant in \cite{Le-URJEPA}, where
the intrinsic dimension $n$ of the embeddings is not known exactly and has to be
estimated from the data.

We remark that these statements concern the evaluation of the two quantities and
not the Carleson constants in Theorem~\ref{thm:main}: by Corollary~\ref{cor:log} the
constants for $\Delta_\mu$ and $\Delta^{\log}_\mu$ can differ by a factor as large as
$c_0^2$, in either direction.
\end{remark}

\begin{remark}\label{rem:numerics-smooth}
The same considerations from Remark~\ref{rem:numerics} apply to the smooth square
functions as well. Here, the role of $\mu(B(x,r))/\mu(B(x,2r))$ is played by
\[
\frac{\phi_t*\mu(x)}{\phi_{2t}*\mu(x)}
=2^n\,\frac{\int\phi\bigl((y-x)/t\bigr)\,d\mu(y)}{\int\phi\bigl((y-x)/(2t)\bigr)\,d\mu(y)},
\]
so that $\Delta^{\log}_{\mu,\phi}(x,t)=\log\bigl(\phi_t*\mu(x)/\phi_{2t}*\mu(x)\bigr)$ is
$n\log2$ plus the logarithm of a ratio of two kernel weighted masses, in which no
power of $t$ appears and the dimension enters only through the additive constant. By
Lemma~\ref{lem:confinement-smooth} this ratio lies in the fixed interval
$[\phi(1)/(C_\phi c_0^2),\,C_\phi c_0^2/\phi(1)]$ for $x\in\supp(\mu)$ and all $t>0$;
unlike the ratio of masses of balls it is not bounded by $1$, since $t\mapsto\phi_t*\mu(x)$
need not be monotone. The invariance under $\mu\mapsto a\mu$ and under dilations
holds verbatim, for an empirical measure the ratio is a quotient of two kernel weighted
sums in which $N$ cancels, and the derivative version
$\widetilde\Delta^{\log}_{\mu,\phi}=\widetilde\Delta_{\mu,\phi}/D_{\mu,\phi}$ has the same
properties. The same considerations also apply to higher order square functions.
\end{remark}

\section{Proof of Proposition~\ref{prop:higher}}\label{app:higher}

We first prove the uniform bound on the higher order square functions.

\begin{lemma}\label{lem:bounded-higher}
Let $\mu$ be $n$-AD-regular with constant $c_0$ and $\mu(\R^d)=\infty$, and let $\phi$
be as in Theorem~\ref{thm:smooth}. Then the function
$t\mapsto D_{\mu,\phi}(x,t)$ is $C^\infty$ on $(0,\infty)$ for each $x$, and for every
$j\ge1$ there is a constant $C_j$ depending on $j$, $\phi$, $n$ and $c_0$ such that
\[
|\widetilde\Delta^j_{\mu,\phi}(x,t)|\le C_j
\qquad\text{for all } x\in\supp(\mu),\ t>0.
\]
\end{lemma}

\begin{proof}
Differentiating $\phi_t(z)=t^{-n}\phi(z/t)$ $j$ times in $t$ gives
$\partial^j_\phi(z,t)=t^j\partial_t^j\phi_t(z)=t^{-n}\psi_j(z/t)$, where $\psi_j$ is a
finite linear combination of the functions $z^\alpha\partial^\alpha\phi(z)$ with
$|\alpha|\le j$. For $\phi(z)=e^{-|z|^{2N}}$ the derivatives of $\phi$ are polynomials
times $\phi$, and for $\phi(z)=(1+|z|^2)^{-a}$ one has
$|\partial^\alpha\phi(z)|\lesssim(1+|z|)^{-2a-|\alpha|}$; in both cases
$|\psi_j(z)|\le C(1+|z|)^{-2a}$ with $a>n/2$ (in the Gaussian case this holds for
every $a$, and we fix $a=n$). The differentiation under the integral sign is
justified by this decay together with \eqref{eq:ADR}. For the bound, note that
$\psi_j$ itself need not be radial or monotone, but it is majorized by the radial
nonincreasing function $\eta(z):=C(1+|z|)^{-2a}$, and the annulus argument in the
proof of Lemma~\ref{lem:confinement-smooth} applied to $\eta$ gives
$|\widetilde\Delta^j_{\mu,\phi}(x,t)|\le\int t^{-n}\eta((y-x)/t)\,d\mu(y)
\le c_0\bigl(\eta(0)+\sum_{k\ge1}2^{kn}\eta(2^{k-1})\bigr)=:C_j$, which is finite
since $2a>n$.
\end{proof}

\begin{proof}[Proof of Proposition~\ref{prop:higher}]
Since $F\in C^1$ and $\lambda^{-1}\le|F'|\le\lambda$ on $J_\phi$, the mean value
theorem gives that $F$ is bi-Lipschitz on $J_\phi$ with constant $\lambda$ (see
Remark~\ref{rem:scope}), so Theorem~\ref{thm:smooth} applies to $F$. For $k=1$ the
proposition is Theorem~\ref{thm:smooth}, so we assume $k\ge2$. Note that the lower
bound $|F'|\ge\lambda^{-1}$ is used only in the sufficiency direction, at the very last
step, while the necessity direction uses only the upper bounds on the derivatives of
$F$. By Lemma~\ref{lem:confinement-smooth}, $D_{\mu,\phi}(x,t)\in J_\phi$ for $x\in\supp(\mu)$,
so the derivatives of $F$ in \eqref{eq:faa} are evaluated in $J_\phi$ and are bounded
by $\lambda$ and $\Lambda$. Together with Lemma~\ref{lem:bounded-higher} this gives,
for $1\le j\le k$ and all $x\in\supp(\mu)$, $t>0$,
\begin{equation}\label{eq:higher-pointwise}
|\widetilde\Delta^{F,j}_{\mu,\phi}|\le C
\qquad\text{and}\qquad
|\widetilde\Delta^{F,j}_{\mu,\phi}|
\le\lambda\,|\widetilde\Delta^j_{\mu,\phi}|+C\sum_{i<j}|\widetilde\Delta^i_{\mu,\phi}|,
\end{equation}
where in the second inequality we bounded each product in \eqref{eq:faa}, which
contains at least one factor $\widetilde\Delta^i_{\mu,\phi}$ with $i<j$, by a constant
times that factor.

Suppose first that $\mu$ is uniformly $n$-rectifiable. By
\cite[Proposition 1.3]{CGLT}, each $\widetilde\Delta^i_{\mu,\phi}$, $1\le i\le k$,
satisfies the Carleson condition \eqref{eq:carleson} with $\Delta_\mu$ replaced by
$\widetilde\Delta^i_{\mu,\phi}$. Squaring the second inequality in
\eqref{eq:higher-pointwise} with $j=k$ and integrating gives \eqref{eq:carlesonFk}.

Suppose now that \eqref{eq:carlesonFk} holds. Fix $x\in\supp(\mu)$ and write
$g(t)=F(D_{\mu,\phi}(x,t))$ and $Y_j(t)=t^jg^{(j)}(t)=\widetilde\Delta^{F,j}_{\mu,\phi}(x,t)$.
Since $t\partial_t(t^{j-1}g^{(j-1)})=(j-1)t^{j-1}g^{(j-1)}+t^jg^{(j)}$, we have
\begin{equation}\label{eq:ode}
t\,Y_{j-1}'(t)=(j-1)\,Y_{j-1}(t)+Y_j(t),\qquad 2\le j\le k.
\end{equation}
The function
\[
Z(t):=-\int_t^\infty\Bigl(\frac{t}{t'}\Bigr)^{j-1}Y_j(t')\,\frac{dt'}{t'}
\]
is well defined and bounded, since $Y_j$ is bounded by \eqref{eq:higher-pointwise} and
$\int_t^\infty(t/t')^{j-1}\,dt'/t'=1/(j-1)$, and a direct computation shows that $Z$
satisfies \eqref{eq:ode} as well. Hence $W=Y_{j-1}-Z$ satisfies $tW'=(j-1)W$, so
$W(t)=Ct^{j-1}$, and since $W$ is bounded on $(0,\infty)$ we get $C=0$. Thus
\begin{equation}\label{eq:descent}
\widetilde\Delta^{F,j-1}_{\mu,\phi}(x,t)
=-\int_t^\infty\Bigl(\frac{t}{t'}\Bigr)^{j-1}\widetilde\Delta^{F,j}_{\mu,\phi}(x,t')\,\frac{dt'}{t'},
\qquad 2\le j\le k.
\end{equation}
We claim that if $\widetilde\Delta^{F,j}_{\mu,\phi}$ satisfies the Carleson condition
with constant $c_j$, then so does $\widetilde\Delta^{F,j-1}_{\mu,\phi}$, with a constant
depending on $c_j$, $j$, $C$ and $c_0$. Fix a ball $B(x_0,R)$ with $x_0\in\supp(\mu)$
and split the integral in \eqref{eq:descent} at $t'=R$. For the part $t'>R$ we use
$|\widetilde\Delta^{F,j}_{\mu,\phi}|\le C$, which gives the bound
$C(t/R)^{j-1}/(j-1)$; its square integrates to
\[
\int_0^R\!\!\int_{B(x_0,R)}\frac{C^2}{(j-1)^2}\Bigl(\frac tR\Bigr)^{2(j-1)}d\mu(x)\,\frac{dt}{t}
=\frac{C^2\,\mu(B(x_0,R))}{2(j-1)^3}\le\frac{C^2c_0}{2(j-1)^3}\,R^n,
\]
by \eqref{eq:ADR}, which applies since $\diam(\supp(\mu))=\infty$. For the part
$t'\le R$, in the variable $s=\log t$ the function
$t\mapsto\int_t^R(t/t')^{j-1}\widetilde\Delta^{F,j}_{\mu,\phi}(x,t')\,dt'/t'$ equals
$\int_\R K(s-\sigma)\,h(\sigma)\,d\sigma$ with
$h(\sigma)=\widetilde\Delta^{F,j}_{\mu,\phi}(x,e^{\sigma})\,\chi_{\{e^\sigma\le R\}}$ and
$K(u)=e^{(j-1)u}\chi_{\{u\le0\}}$, that is, the convolution of $h$ with $K$, and
$\|K\|_{L^1(\R)}=1/(j-1)$. By Young's
inequality, for each $x$,
\[
\int_0^R\Bigl|\int_t^R\Bigl(\frac{t}{t'}\Bigr)^{j-1}\widetilde\Delta^{F,j}_{\mu,\phi}(x,t')\,\frac{dt'}{t'}\Bigr|^2\frac{dt}{t}
\le\frac1{(j-1)^2}\int_0^R|\widetilde\Delta^{F,j}_{\mu,\phi}(x,t)|^2\,\frac{dt}{t},
\]
and integrating in $x\in B(x_0,R)$ gives the bound $(j-1)^{-2}c_jR^n$. This proves
the claim. Applying it for $j=k,k-1,\dots,2$ we obtain the Carleson condition for
$\widetilde\Delta^{F,1}_{\mu,\phi}=\widetilde\Delta^F_{\mu,\phi}=F'(D_{\mu,\phi})\,
\widetilde\Delta_{\mu,\phi}$, and Theorem~\ref{thm:smooth} gives that $\mu$ is
uniformly $n$-rectifiable. It is here that the lower bound $|F'|\ge\lambda^{-1}$ on
$J_\phi$ is needed: it is what allows one to pass from the Carleson condition for
$F'(D_{\mu,\phi})\widetilde\Delta_{\mu,\phi}$ to the one for $\widetilde\Delta_{\mu,\phi}$.
\end{proof}

\section{Proof of Proposition~\ref{prop:bounded}}\label{app:bounded}

We first prove the bound on the total mass of $\mu$ used in
Section~\ref{sec:bounded}. Note that
\eqref{eq:ADR} at the radius $r=R_0$ bounds $\mu(B(x,R_0))$, but the ball $B(x,R_0)$ may
omit points of $\supp(\mu)$ at distance exactly $R_0$ from $x$, and these may carry
positive mass. We use a covering argument instead.

\begin{lemma}\label{lem:totalmass}
Let $\mu$ be $n$-AD-regular with constant $c_0$ and $R_0=\diam(\supp(\mu))<\infty$.
Then
\[
c_0^{-1}R_0^n\;\le\;\mu(\R^d)\;\le\;c_1R_0^n=5^dc_0R_0^n.
\]
\end{lemma}

\begin{proof}
The lower bound is \eqref{eq:ADR} at $r=R_0$ for any $x\in\supp(\mu)$. For the upper
bound, let $\{z_i\}_{i=1}^N\subset\supp(\mu)$ be a maximal family of points with
$|z_i-z_j|\ge R_0/2$ for $i\ne j$. By maximality every $y\in\supp(\mu)$ satisfies
$|y-z_i|<R_0/2$ for some $i$, so $\supp(\mu)\subset\bigcup_iB(z_i,R_0/2)$. The balls
$B(z_i,R_0/4)$ are pairwise disjoint and contained in $B(z_1,\tfrac54R_0)$, so
comparing Lebesgue measures gives $N(R_0/4)^d\le(\tfrac54R_0)^d$, that is, $N\le5^d$.
Since $R_0/2\le R_0$ is admissible in \eqref{eq:ADR},
\[
\mu(\R^d)\;\le\;\sum_{i=1}^N\mu\bigl(B(z_i,R_0/2)\bigr)\;\le\;5^d\,c_0\,(R_0/2)^n
\;\le\;c_1R_0^n.
\qedhere
\]
\end{proof}

\begin{proof}[Proof of Proposition~\ref{prop:bounded}]
By regimes (i) and (ii) of Section~\ref{sec:bounded}, both densities lie in $J_1$ for every
$x\in\supp(\mu)$ and every $0<r\le R_0$, so the pointwise comparability
\eqref{eq:comparable} holds on that range of scales with $\lambda_1$ in place of
$\lambda$; for $F=\log$ the mean value theorem on $J_1$ gives
$\lambda_1=\max(c_1,2^nc_0)=5^dc_0$, since $n<d$.

\emph{Proof of (a).} Assume \eqref{eq:carlesonF}. Fix a ball $B(x_0,R)$ with
$x_0\in\supp(\mu)$; we verify \eqref{eq:carleson} for it. If $R\le R_0$, comparability on
$(0,R]$ gives
\eqref{eq:carleson} with constant $\lambda_1^2c$ directly. If $R>R_0$, split the
$r$-integral at $R_0$. On $(0,R_0]$ the same comparison applies. On $(R_0,R]$ we use only
the trivial bound: for $r>R_0$ and any $x$,
\[
|\Delta_\mu(x,r)|\le \Dm(x,r)=\frac{\mu(\R^d)}{r^n}\le c_1\Bigl(\frac{R_0}{r}\Bigr)^{n}
\]
by Lemma~\ref{lem:totalmass}, so that
\[
\int_{R_0}^{R}\!\!\int_{B(x_0,R)}|\Delta_\mu|^2\,d\mu\,\frac{dr}{r}
\;\le\;\mu(\R^d)\,c_1^2\int_{R_0}^{\infty}\Bigl(\frac{R_0}{r}\Bigr)^{2n}\frac{dr}{r}
\;=\;\frac{c_1^2\,\mu(\R^d)}{2n}
\;\le\;\frac{c_1^3}{2n}\,R_0^n\;\le\;\frac{c_1^3}{2n}\,R^n.
\]
Hence \eqref{eq:carleson} holds with constant $\lambda_1^2c+c_1^3/(2n)$, and
Theorem~\ref{thm:CGLT} gives that $\mu$ is uniformly $n$-rectifiable. Note that the
integral over $(R_0,R]$ was estimated for $\Delta_\mu$ and not for $\Delta^F_\mu$. This
is the reason why no hypothesis on $F$ near $0$ is needed in (a).

\emph{Proof of (b).} Assume $\mu$ uniformly rectifiable, so \eqref{eq:carleson} holds
with some $c'$ by Theorem~\ref{thm:CGLT}. Fix $B(x_0,R)$, $x_0\in\supp(\mu)$. If
$R\le R_0$ we conclude by comparability as before. If $R>R_0$, split at $R_0$. On
$(0,R_0]$, comparability and \eqref{eq:carleson} give the bound $\lambda_1^2c'R^n$. On
$(R_0,R]$ both densities equal $\mu(\R^d)r^{-n}$ and $\mu(\R^d)(2r)^{-n}$ respectively,
so with the substitution $t=\mu(\R^d)\,r^{-n}$ (hence $\frac{dt}{t}=-n\frac{dr}{r}$, and
$t$ ranges over $[\varepsilon,\,\mu(\R^d)R_0^{-n}]$ with
$\varepsilon:=\mu(\R^d)R^{-n}$),
\[
\int_{R_0}^{R}\!\!\int_{B(x_0,R)}|\Delta^F_\mu|^2\,d\mu\,\frac{dr}{r}
\;=\;\frac{\mu(\R^d)}{n}
\int_{\varepsilon}^{\mu(\R^d)R_0^{-n}}\bigl|F(t)-F(2^{-n}t)\bigr|^2\,\frac{dt}{t}.
\]
Since $\mu(\R^d)R_0^{-n}\le c_1$ by Lemma~\ref{lem:totalmass} and
$\mu(\R^d)=\varepsilon R^n$, the right-hand side is at most
\[
\frac{R^n}{n}\cdot\varepsilon
\int_{\varepsilon}^{c_1}\bigl|F(t)-F(2^{-n}t)\bigr|^2\,\frac{dt}{t}
\;\lesssim\;R^n
\]
by \eqref{eq:tail}. If $|F(t)-F(2^{-n}t)|\le M$ for all $t\in(0,c_1]$, then
$\varepsilon\int_\varepsilon^{c_1}(\cdot)\,\frac{dt}{t}\le
M^2\,\varepsilon\log(c_1/\varepsilon)$, which is bounded uniformly in
$\varepsilon\in(0,c_1]$; thus \eqref{eq:tail} holds. If $F$ is continuous on
$(0,\infty)$, then $t\mapsto F(t)-F(2^{-n}t)$ is bounded on every compact subinterval
of $(0,\infty)$, so boundedness near $0$ already gives boundedness on $(0,c_1]$; note
that for a merely Borel $F$ this is not the case, and boundedness near $0$ alone does
not imply \eqref{eq:tail}. For $F=\log$ the integrand is the
constant $(n\log 2)^2$ and the same computation applies; alternatively, for $r>R_0$ one
has $\Delta^{\log}_\mu(x,r)\equiv n\log 2$ and the tail is bounded by
$\mu(\R^d)(n\log2)^2\log(R/R_0)\lesssim R_0^n\log(R/R_0)\lesssim R^n$, since
$u\mapsto u^n\log(1/u)$ is bounded on $(0,1]$.
\end{proof}

\end{document}